\documentclass[11pt]{article}

\usepackage{amssymb,amsmath,amsthm}
\usepackage{graphicx}

\def\Pr{{\mathbb P}}
\def\R{{\mathbb R}}
\def\Rp{{\mathbb R_+^2}}
\def\eps{{\varepsilon}}
\def\ind{{\mathbf 1}}

\def\Poiss{\operatorname{Poiss}}

\newtheorem{lem}{Lemma}
\newtheorem{conj}[lem]{Conjecture}
\newtheorem{thm}[lem]{Theorem}

\newtheorem{cor}[lem]{Corollary}
\newtheorem{rem}[lem]{Remark}
\newtheorem{defi}[lem]{Definition}

\title{Isoperimetric problem for  exponential measure \\ on the plane with $\ell_1$-metric
}
\author{Marta Strzelecka}
\date{10.06.2016}

\begin{document}

\maketitle

\begin{abstract}
	We give a solution to the isoperimetric problem for the exponential measure   on the plane with the $\ell_1$-metric. As it turns out, among all sets of a given measure, the simplex or its complement (i.e. the ball in the $\ell_1$-metric or its complement) has the smallest boundary measure. The proof is based on a symmetrisation (along the sections of equal $\ell_1$-distance from the origin).
\end{abstract}

\section{Introduction and main result}

	For  a metric space $(X,d)$ equipped with a Borel  measure $\mu$ we define the boundary measure $\mu^+$  of a Borel set $A$ as
	\[
		\mu^+(A):= \liminf_{h\to 0+} \frac{\mu(A^h)-\mu(A)}{h},
	\]
where $A^h:= \{x\in X: \ \exists y\in A \ \  d(x,y)<h \}$ is an $h$-neighbourhood of $A$ with respect to $d$. It is interesting to study the isoperimetric problem: among all sets of a given measure  find a set with the smallest boundary measure. In other words, we want to find a set which measure grows the slowest among all sets of a given measure. Such a set is said to be  \emph {extremal}.

	This problem seems to be  difficult in general and the solution to it is known  only in a few cases. 
	If $\mu $ is the Lebesgue measure in the $n$-dimensional Euclidean space, then  balls are extremal sets. This follows for example by the Brunn-Minkowski inequality and can be proven in many other ways (see for example \cite[Section 2]{Os}).  L{\'e}vy \cite{Lev} and Schmidt \cite{Sch} proved that the extremal sets  with respect to the Haar measure   on the $n$-dimensional sphere equipped with the geodesic metric are  balls in the geodesic metric, i.e. the intersections of half-spaces  in $\R^{n+1}$ with the sphere. 
	
	Another example of the full solution to the isoperimetric problem is the Gaussian measure in the Euclidean space $\R^n$, i.e. the product measure with the density $(2\pi)^{-n/2} e^{-|x|^2/2}$, where $|\cdot|$ is the Euclidean norm in $\R^n$. Borell \cite{Bo} and Sudakov with Tsirelson \cite{ST} proved that in this case half-spaces $\{x: \langle x,u \rangle \ge \lambda \}$ are extremal. As Bobkov and Houdr{\'e} proved in \cite{BH}, on the real line this result can be generalized into the case of an arbitrary symmetric log-concave measure.  Bobkov \cite{Bo} also studied  the isoperimetric problem in the product  metric space $(X^n, d_{\sup})$ equipped with a  product probability measure, where $d_{\sup}(x,y):=\sup_{i\le n}d(x_i, y_i)$. In this case, if the extremal sets in $X^2$ are of the form $A\times X$ and $A$ are extremal in $X$, then $A\times X^{n-1}$ are extremal in $X^n$.
	
	The discrete version of the isoperimetric problem on the cube $\{ -1,1\}^n$ (with the uniform measure and the Hamming distance) was considered by Harper in \cite{Har}. Roughly speaking, he showed that balls in the Hamming distance are extremal.  This was generalized to sets $\{ 0,1, \ldots, d-1\}^n$ (instead of $\{ -1,1\}^n$) by Wang and Wang in \cite{WW}.

	It should be noted, that once we know the solution to the isoperimetric problem, we can obtain  concentration properties for the measure $\mu$ (see for example Chapter 2.1 of \cite{Led}). However, it is probably the most difficult way to derive concentration inequalities, since it relies on  finding the exact value of the isoperimetric function (and the sets which achieve the smallest boundary measure), not only a reasonable estimate on it.
	
	In this note we will find the extremal sets in the case of the exponential measure on the plane with $\ell_1$-metric. Let $\nu$ be the product exponential measure on $\R_+^n  = [0,\infty)^n$, i.e. the measure with the density
	\[
  		e^{-\sum_{i=1}^n x_i}\ind_{x\in \R_+^n},
	\]
 and let  $B_1^n$ be the unit ball  in  the $\ell_1$-distance (centred at the origin). 
 
 \begin{defi}\label{def_B_A}
 	For a Borel set $A\subset \R$ we define the set $B_A$ by a formula
 	\begin{displaymath}
		B_A := \left\{ \begin{array}{ll}
		tB_1^n & \textrm{ if \ } \nu(A)\ge \frac 12\\
		\R_+^n\setminus tB_1^n & \textrm{ if \ }\nu(A) <\frac 12
\end{array} \right. ,
\end{displaymath}
 	where $t$ is the unique positive number for which $\nu(B_A)=\nu(A)$. We call such a number $t$ the radius of $B_A$.
 \end{defi}

	In other words $B_A$ is a simplex or a complement of a simplex, and has the same measure as $A$.  As  will be clear from  Lemma \ref{lem_trap} below, out of these two sets we pick   the one of smaller boundary measure.

Our main result is the following theorem, which  states that among all Borel sets of a given measure, a simplex or its complement has the smallest boundary measure. Unfortunately, we are able to give the complete proof only in the case $n=2$, but a part of our reasoning works also for a general $n$. 

\begin{thm} \label{main_theorem}
	If $A$ is a Borel set in $\R_+^2$, then 
	\begin{equation}\label{isoper}
		\nu^+(A)\ge \nu^+(B_A).
	\end{equation}
\end{thm}

We call a Borel subset $A \subset\R^n$ $1$-unconditional if $x\in A$ implies that $(\eps_1 x_1, \ldots, \eps_n x_n) \in A$ for every choice of signs $(\eps_i)_{i=1}^n \in \{-1,1 \}^n$. Note that if $A$ is such a set and $x\in A^h\cap \R_+^n$, then there exists $y\in A\cap \R_+^n$ such that $\|x-y\|_1< h$. This together with the previous theorem implies the following isoperimetric inequality.

\begin{cor}
	Let $\mu$ be the symmetric exponential measure on the plane, i.e. the measure with  density $\frac{1}4 e^{-|x|-|y|}$. Then, among all $1$-unconditional Borel sets $A$,   a ball or its complement has the smallest boundary measure.
\end{cor}

However, the balls are not extremal sets for the symmetric exponential measure on the plane. An example is the set $A:= \{x+y\le 3 \}$, which boundary measure is smaller than the boundary measure of the simplex of the same measure.

We believe that  Theorem \ref{main_theorem} holds also in higher dimensions.

\begin{conj}
	If $A$ is a Borel set in $\R_+^n$, then 
	\begin{equation*}
		\nu^+(A)\ge \nu^+(B_A).
	\end{equation*}
\end{conj}

	The organization of this paper is the following. First   we prove that among all \emph{trapezoids} (i.e. the sets of the form $RB_1^n \setminus rB_1^n$ for $0\le r<R\le \infty$) of a given measure the simplex $tB_1^n$ or its complement  has the smallest boundary measure (see Lemma \ref{lem_trap}). Then we show that in order to prove the isomerimetric inequality \eqref{isoper}, it suffices to consider connected compact sets of non-empty interior. After that we restrict our attention to the case $n=2$.  We do  symmetrisations, which   lead us  from a given connected compact set $A$ to a  \emph{trapezoid} of the same measure (see Lemma \ref{lemat_red_1} and proof of Theorem \ref{main_theorem}).

	\section{Proof of Theorem \ref{main_theorem}}
	
	Note that $\nu^+(B_A)$ depends on $\nu (A)$ (or equivalently on $\nu(B_A)$) in a smooth way, excluded the case $\nu(A)=\frac{1}2$. Indeed, one can  calculate that
	\begin{equation*}
		\nu(tB_1^n) = \nu(\R_+^n \cap tB_1^n) = C_n \int_0^t e^{-x}x^{n-1}dx
	\end{equation*}
	and therefore
	\begin{equation*}
		\nu^+(tB_1^n) = \nu^+(\R_+^n \cap tB_1^n) = C_ne^{-t}t^{n-1} = \nu^+(\R_+^n \setminus tB_1^n) = \nu^+((tB_1^n)^c),
	\end{equation*}
	where
	\begin{equation*}
		C_n= \left(\int_0^\infty e^{-x}x^{n-1}dx\right)^{-1} = \frac{1}{(n-1)!}.
	\end{equation*}
	Moreover, if $A$ is a finite sum of $\ell_1$-balls, we can write the true limit in the definition of $\mu^+(A)$. We use these facts to deduce the following technical lemma.
	
	\begin{lem}
	
	\begin{enumerate}	
	\item[(i)]
		Inequality \eqref{isoper} holds for every Borel set $A$ of measure at least $\frac 12$ if and only if for  every finite union $B$ of $\ell_1$-balls, such that $\nu(B)\ge \frac 12$,  we have
		\begin{equation}\label{isoper_neigh1}
			\varphi^{-1}\big(\nu(B^h)\big) \ge \varphi^{-1}\big(\nu(B)\big) +h, \qquad \mbox{for all } h>0,
		\end{equation}
		where $\varphi(t) := \nu(tB_1^n) =  C_n\int_0^t e^{-x}x^{n-1}dx$. Moreover,  inequalities \eqref{isoper} for finite unions of $\ell_1$-balls of measure at least $\frac 12$ and \eqref{isoper_neigh1} for finite unions of $\ell_1$-balls of measure at least $\frac 12$   are equivalent.

	\item[(ii)]
		Inequality \eqref{isoper} holds for every Borel set $A$ of measure less than $\frac 12$ if and only if for  every finite union  $B $ of $\ell_1$-balls, such that $\nu(B^h)<\frac 12$,  we have
		\begin{equation}\label{isoper_neigh2}
			\psi^{-1}\big(\nu(B^h)\big) \le \psi^{-1}\big(\nu(B)\big) -h,  \qquad \mbox{for all } h>0,
		\end{equation}
		where $\psi(t) := \nu(\R^n\setminus tB_1^n) = C_n\int_t^{\infty} e^{-x}x^{n-1}dx$. Moreover, inequalities \eqref{isoper} for finite unions of $\ell_1$-balls of measure less than $\frac 12$ and \eqref{isoper_neigh1} for finite unions of $\ell_1$-balls of measure less than $\frac 12$   are equivalent.
		\end{enumerate}	
	\end{lem}
	
	\begin{proof}
		We will only show (i), since the proof of (ii) is similar.
		
		Assume first that \eqref{isoper}  holds for Borel sets of measure at least $\frac 12$. To prove inequality \eqref{isoper_neigh1} let us introduce the function $h\mapsto \varphi^{-1}(\nu(B^h))$ and note that \eqref{isoper} for $B^h$ (which is also a finite union of balls) implies that the  derivative of this function is bounded from below by $1$. Note that we use \eqref{isoper} only for finite unions of balls, so  we also proved  the second part of (i).
		
		Now suppose that \eqref{isoper_neigh1} holds for finite unions of balls with measure at least $\frac 12$. It is obvious that \eqref{isoper_neigh1} for a set $A$ implies \eqref{isoper} for this $A$, so we only have to show  \eqref{isoper_neigh1} for the set $A$ instead of $B$. Note that for $r>0$ the set $A^r$ is open and therefore it can be represented as a countable union of balls $\bigcup \mathcal{U}_{r}$. Let $\mathcal{U}_{r,m}$ be a subfamily of the family  $\mathcal{U}_r$ containing the first $m$ balls (so that $\mathcal{U}_{r,m+1}\setminus \mathcal{U}_{r,m}$ contains a single ball). Then, by the continuity and the monotonicity of $\varphi$, and the inequality \eqref{isoper_neigh1} for $\bigcup\mathcal{U}_{r,m}$, we have
		\begin{multline*}
			\varphi^{-1}\big(\nu(A^{r+h})\big) \ge \varphi^{-1}\bigg(\nu\Big( \Big(\bigcup \mathcal{U}_{r,m} \Big)^{h}\Big)\bigg) \ge \varphi^{-1} \Big(  \nu\big(\bigcup \mathcal{U}_{r,m} \big) \Big) + h \mathop{\longrightarrow}^{m\to \infty} \varphi^{-1}\big(\nu(A^{r})\big) + h  
			\\
			 \ge  \varphi^{-1}\big(\nu({A})\big) + h
		\end{multline*}
		for sufficiently large $n$ (depending on $r>0$) . We take $r\to 0$ on the left-hand side of this estimate to get \eqref{isoper_neigh1} for $\overline{A^h}$. In particular, for any $h>\eps>0$ we have \eqref{isoper_neigh1} for $\overline{A^{h-\eps}}$, so $\varphi^{-1}\big(\nu(A^h)\big) \ge \varphi^{-1}\big(\nu(A)\big) +h-\eps$. We take $\eps \to 0$ to get \eqref{isoper_neigh1} for $A$.  This finishes the proof.
	\end{proof} 
	
	\begin{cor}\label{cor_fin_balls}
		It suffices to prove the isoperimetric inequality \eqref{isoper} for finite unions of $\ell_1$-balls.
	\end{cor}

	We start the main part of the proof of the isoperimetric inequality by showing that the simplex or its complement is the set growing most slowly among all \emph{trapezoids} of a given measure.
	
	\begin{lem}\label{lem_trap}
	The isoperimetric inequality \eqref{isoper} holds for sets $A$ of the form $\{x\in \R_+^n: a<\|x\|_1<b \}$, where $0\le a<b \le \infty$.
\end{lem}

\begin{proof}
	 Let $c$ be the radius of the set $B_A$ (see Definition \ref{def_B_A}). We consider three cases.
	
	{\bf Case 1.} Assume  $a=0$ or $b=\infty$. To prove the isoperimetric inequality in this case we need only to prove that if $\int_0^x e^{-t} t^{n-1}dt =\int_y^\infty e^{-t} t^{n-1}dt < \frac{1}2\int_0^\infty e^{-t} t^{n-1}dt $, then $e^{-x}x^{n-1}\ge e^{-y}y^{n-1}$ (In the other case we can consider the complements of these sets).  Note that this means that in the definition of  $B_A$, among the simplex and the complement of the simplex, we always pick   the set of smaller boundary measure. 
	
	Let us first show that the condition $\int_y^\infty e^{-t} t^{n-1}dt <\frac{1}2 \int_0^\infty e^{-t} t^{n-1}dt $ implies that $y\ge n-1$. To this end we only have to show that $\int_{n-1}^\infty e^{-t} t^{n-1}dt \ge \frac{1}2 \int_0^\infty e^{-t} t^{n-1}dt $. Integration by parts yields 
	\begin{equation*}
		\frac{\int_{n-1}^\infty e^{-t} t^{n-1}dt }{ \int_0^\infty e^{-t} t^{n-1}dt} 
		 = e^{-(n-1)} \left( \frac{(n-1)^{n-1}}{ (n-1)!} + \frac{(n-1)^{n-2}}{(n-2)!} +\ldots + \frac{n-1}{1!} +1 \right),
	\end{equation*}
  so we only  have to show that $\Pr (\Poiss (k) \le k) \ge \frac{1}2$, where $\Poiss(\lambda)$ is the random variable of Poisson distriubution with parameter $\lambda$. Due to \cite[Theorem 1]{Cho}, the smallest integer $l$ for which $\Pr (\Poiss(\lambda) \le l) \ge \frac 12$ satisfies $\lambda - \log 2 \le l < \lambda +\frac 13$.
   This implies that $\Pr (\Poiss (k) \le k) \ge \frac{1}2$ and therefore  finishes the proof of the inequality $y\ge n-1$.

	One can easily check that the function $t^{n-1}e^{-t}$ is decreasing on the half-line $[n-1, \infty)$, so if $n-1\le x\le y$, then $e^{-x}x^{n-1}\ge e^{-y}y^{n-1}$ and we are done. Otherwise $x\le n-1\le y$ (since $x\le y$ and $n-1\le y$). Let us consider this case now. Note that the equation  $\int_0^x e^{-t} t^{n-1}dt =\int_y^\infty e^{-t} t^{n-1}dt  $ determines $y$ as a function of $x$ and $e^{-x}x^{n-1}=-y'e^{-y}y^{n-1}$. For $x=0$ the inequality  $e^{-x}x^{n-1}\ge e^{-y}y^{n-1}$ holds (and is in fact an equality, since $y(0)=\infty$). For $x=n-1$ the  function $e^{-x}x^{n-1}$ attains its maximum on $[0,\infty]$, so $e^{-x}x^{n-1}\ge e^{-y}y^{n-1}$ if $x=n-1$. Therefore it suffices to check whether $e^{\frac{x-y}{n-1}} \frac{y}x \le 1$ for every $x \le n-1$ at which  the derivative of $e^{-x}x^{n-1} - e^{-y}y^{n-1}$ vanishes. This derivative is equal to 
	\begin{multline*}
		e^{-x}x^{n-2}\left( -x+(n-1)  \right) - y'e^{-y}y^{n-2}\left( -y+(n-1) \right)  
		\\
		= e^{-x}x^{n-2} \left( -x+(n-1) +  \frac xy(n-1) -x    \right),
	\end{multline*}
	so we should check  values of $x$ satisfying $y(x)=\frac{x(n-1)}{2x-(n-1)}$. In particular, these values are greater than $\frac{n-1}2$. Note that if $y(x)=\frac{x(n-1)}{2x-(n-1)}$, then $e^{\frac{x-y}{n-1}} \cdot \frac{y}x = e^{-2\lambda\frac{1-\lambda}{2\lambda -1}} / (2\lambda -1)$, where $\lambda:=\frac{x}{n-1} \in (\frac{1}2, 1]$. The derivative of $e^{-2\lambda\frac{1-\lambda}{2\lambda -1}} / (2\lambda -1)$ is equal to $4(1-\lambda)^2 e^{-2\lambda\frac{1-\lambda}{2\lambda -1}} / (2\lambda -1)^{3} \ge 0$, so this function is non-decreasing and therefore less than its value in $1$ (for $\lambda \in(\frac 12,1]$). Hence the inequality  $e^{-x}x^{n-1} \ge e^{-y}y^{n-1}$ holds for $x$ such that $y(x)=\frac{x(n-1)}{2x-(n-1)}$ and the claim is proved.

	{\bf Case 2.} Assume $\nu(A) \le \frac 12$. Then we have
	\begin{align} \label{eq_case2}
		\int_a^b e^{-t}t^{n-1}dt  = \int_c^\infty e^{-t}t^{n-1}dt.
	\end{align}
	For a fixed $b>0$ this equality determines $a$ as a function of $c$ and $a'e^{-a}a^{n-1} = e^{-c}c^{n-1}$. If $c$ is such that $a(c)=0$, then the inequality we want to prove, $e^{-a}a^{n-1}+e^{-b} b^{n-1}\ge e^{-c}c^{n-1}$, holds as we proved in Case 1. Therefore it suffices to show that the derivative of $e^{-a}a^{n-1}- e^{-c}c^{n-1}$ as a function of $c$ is non-negative. Integration by parts of  \eqref{eq_case2} yields
	\begin{equation*}
		(n-1)\int_a^b e^{-t}t^{n-2}dt   - (n-1)\int_c^\infty e^{-t}t^{n-2}dt =- e^{-a}a^{n-1}+e^{-b} b^{n-1}+ e^{-c}c^{n-1},
	\end{equation*}
	so the derivative of $e^{-a}a^{n-1}- e^{-c}c^{n-1}$  is equal to 
	\begin{equation*}
		(n-1)(a'e^{-a}a^{n-2}- e^{-c}c^{n-2}) = (n-1) e^{-c}c^{n-2} \left(\frac ca -1\right).
	\end{equation*}
  Since $c\ge a$,  the derivative we consider is indeed non-negative.
	
	{\bf Case 3.} Assume $\nu(A) \ge \frac 12$. We proceed similarly as in Case 2. Since $\nu(A) \ge \frac 12$, we have
	\begin{align*}
		\int_a^b e^{-t}t^{n-1}dt  = \int_0^c e^{-t}t^{n-1}dt.
	\end{align*}
	For a fixed $a>0$ this equality determines $b$ as a function of $c$ and $b'e^{-b}b^{n-1} = e^{-c}c^{n-1}$. For $c$ such that $b(c)=\infty$ the inequality $e^{-a}a^{n-1}+e^{-b} b^{n-1}\ge e^{-c}c^{n-1}$, holds, what we proved in Case~1. Therefore it suffices to prove that the derivative of $(e^{-b} b^{n-1} - e^{-c}c^{n-1})$ is negative. Calculations similar to those carried out in Case 2  show  that this derivative is equal to $(n-1)e^{-c}c^{n-2}(\frac{c}b -1)$, which is negative, since $b>c$.
\end{proof}

		The next lemma allows us to restrict our attention to connected compact sets.	

\begin{lem}\label{simplification_lemma}
	If  for every connected compact set $A$ the inequality
	\begin{equation}\label{aim_comp}
		\nu(A^h)\ge \nu(B_A^h)-Lh^2
	\end{equation}
	holds for every $h\le h_0$ with $L,h_0$ depending on $A$ only, then the isoperimetric inequality \eqref{isoper} holds for every Borel set $A$.
\end{lem}

\begin{proof} Let $A$ be a Borel set of positive measure. Assuming \eqref{aim_comp} for connected bounded Borel sets, we are going to prove \eqref{isoper} for $A$.

	By Corollary \ref{cor_fin_balls} it suffices to prove \eqref{isoper} for finite unions of balls.  Moreover, if we show \eqref{isoper} for $\overline{A}$, the inequality for $A$ will follow (since $\nu(\overline{A}^h) = \nu(A^h)$ and $\nu(\overline{A}) \ge \nu(A)$).   Since \eqref{aim_comp} implies \eqref{isoper}, it suffices to prove \eqref{aim_comp} for a compact set $A$ with finitely many connected components $A_1, \ldots, A_N $, each of non-empty interior. Then for sufficiently small $h_0>0$ the sets $A_j^{h_0}$ are pairwise disjoint, so for any $h\in(0,h_0)$ we have $\nu(A^h)=\sum_{j=1}^M \nu(A_j^h)\ge \sum_{j=1}^M \nu(B_{A_j}^h) - NLh^2 $, because \eqref{aim_comp} holds for $A_i$.  To finish the proof  we use Lemma \ref{lem_comp} (see below),  and an obvious induction.
	\end{proof}

	\begin{lem} \label{lem_comp}
			If $C$ and $D$ are disjoint sets, then $\nu^+(B_{C})+\nu^+( B_{D})\ge \nu^+(B_{C\cup D})$.
	\end{lem}
		
	\begin{proof}
	By $x,y,z$ we denote the radii of $B_C$, $B_D$ and $B_{D\cup C}$ respectively. 
Note that if $\nu(C)\ge \frac{1}2$, then $\nu(D)<\frac 12$ and $\nu(C\cup D) >\frac 12$. Therefore it suffices to consider the following three cases.

	{\bf Case 1.} Assume $\nu(C), \nu(D), \nu({C\cup D}) < \frac{1}2$. Without loss of generality we may assume $x\ge y$. By the definition of the sets $B_C, B_D$ and $B_{C\cup D}$ we have
	\begin{align} \label{assump_comp}
		\int_x^\infty e^{-t}t^{n-1}dt + \int_y^\infty e^{-t}t^{n-1}dt = \int_z^\infty e^{-t}t^{n-1}dt.
	\end{align}
	Integration by parts implies that the inequality $e^{-x}x^{n-1} +e^{-y}y^{n-1} = \nu^+(B_C)+\nu^+(B_D) \ge \nu^+(B_{C\cup D}) = e^{-z}z^{n-1}$ is equivalent to
	\begin{align} \label{claim_comp}
		\int_x^\infty e^{-t}t^{n-2}dt + \int_y^\infty e^{-t}t^{n-2}dt \le \int_z^\infty e^{-t}t^{n-2}dt.
	\end{align}
	We will prove that \eqref{assump_comp} implies \eqref{claim_comp} for $x,y,z\ge0$ such that $x\ge y$. Fix $z\ge 0$. Then  equality \eqref{assump_comp}, determines $x$ as a smooth function of $y$ and $x'e^{-x}x^{n-1}+e^{-y} y^{n-1} =0$. Note that  $x(z)=0$ and thus  for $y=z$ inequality \eqref{claim_comp} holds  (and is in fact an equality). To end the proof in Case~1 we will show that the left-hand side of \eqref{claim_comp} is a decreasing function of $y$ for $y>z$ such that $x(y)> y$. The derivative of  the left-hand side of \eqref{claim_comp} is equal to $-(x'e^{-x}x^{n-2}+ e^{-y} y^{n-2})= e^{-y}y^{n-2}(\frac{y}x-1)$ which is negative for $x>y$. Note that in this case we have used the fact that $\nu(B_C), \nu(B_D), \nu(B_{C\cup D}) \le \frac{1}2 $ only to obtain  \eqref{assump_comp}. 
	
	{\bf Case 2.} Assume $\nu(C), \nu(D) < \frac 12 ,  \  \nu({C\cup D}) \ge \frac{1}2$. Let $v$ be such that $\int_0^z e^{-t}t^{n-1}dt = \int_v^\infty e^{-t}t^{n-1}dt$. Then \eqref{assump_comp}, and consequently \eqref{claim_comp}, holds with $v$ in  place of $z$ and thus $e^{-x}x^{n-1} +e^{-y}y^{n-1} \ge e^{-v}v^{n-1} \ge e^{-z}z^{n-1}$ -- the last inequality holds by Lemma \ref{lem_trap} applied to $a=v$ and $b=\infty$.

	{\bf Case 3.} Assume $\nu(C)<\frac 12, \ \nu(D),   \nu({C\cup D}) \ge \frac{1}2$. Then
	\begin{align} \label{assump_comp2}
		\int_x^\infty e^{-t}t^{n-1}dt + \int_0^y e^{-t}t^{n-1}dt = \int_0^z e^{-t}t^{n-1}dt.
	\end{align}
	Fix  any $x>0$ satisfying $\int_x^\infty e^{-t}t^{n-1}dt \le \frac 12$. Equality \eqref{assump_comp2} determines $z$ as a function of $y$ and $e^{-y}y^{n-1} = z' e^{-z}z^{n-1}$.
	We want to show that $e^{-x}x^{n-1} +e^{-y}y^{n-1}\ge e^{-z}z^{n-1}$. Note that for $y$ such that $\nu(yB_1^n)=\frac 12$ we have $\nu(z(y)B_1^n) \ge \frac 12$ and Case 2 implies that then $e^{-x}x^{n-1} +e^{-y}y^{n-1}\ge e^{-z}z^{n-1}$. Thus it suffices to show that $e^{-y}y^{n-1}- e^{-z}z^{n-1}$ is increasing in $y$. 
	We integrate \eqref{assump_comp2} by parts and get
	\begin{multline}  \label{case3}
		(n-1)\int_x^\infty e^{-t}t^{n-2}dt +e^{-x}x^{n-1} -\left(e^{-y}y^{n-1} - e^{-z}z^{n-1} \right) \\ = (n-1)\int_0^z e^{-t}t^{n-2}dt -  (n-1)\int_0^y e^{-t}t^{n-2}dt.
	\end{multline}
	Therefore we only need to show that the derivative of the right-hand-side of \eqref{case3} is negative. This derivative is equal to
	\begin{align*}
		(n-1) \left(z'e^{-z}z^{n-2} -  e^{-y}y^{n-2} \right) = (n-1)e^{-y}y^{n-2} \left( \frac yz -1 \right),
	\end{align*}
	what is negative since $y<z$.	
	\end{proof}
	
	From now on we assume $n=2$. Otherwise the symmetrisation described below, as well as the final argument,  does not work. Let $T:=\overline{B_1^2\cap \big([0,\infty)^2 \cup (-\infty, 0]^2\big)}$.

	\begin{rem}\label{rem_T}
	Note that  for sets $A$ of the form $rB_1^2 \cap \Rp$ or $\Rp\setminus rB_1^2$ we have $\nu(A+hT)=\nu(A^{h})$, and for any compact set $A$ the inequality $\nu(A+h_nT)\le \nu(A^{h_k})$ holds  for some sequence $(h_k)_{k\ge 0}$ (depending on $A$) which tends to $0$. (We pick a sequence $(h_k)_{k\ge 0}$, because it may happen that 
	\[
	\nu(A+hT) \neq \nu \Big(A+h \big(B_1^2\cap \large([0,\infty)^2 \cup (-\infty, 0]^2 \large) \big) \Big),
	\]
	 but only for finitely many $h>0$.) Therefore in order to prove \eqref{aim_comp}  it suffices to show that every connected compact set $A$ satisfies
	\begin{equation}\label{aim_comp_tr}
		\nu(A+hT) \ge \nu(B_A+hT) -Lh^2 \quad \mbox{for } h\le h_0,
	\end{equation}
	where $L$ is an absolute constant and $h_0$ depends on $A$ only.
	\end{rem}

	\begin{defi}
	For a Borell set $A\subset \R_+^2$  and $t>0$ we define
	\[ 
		 f_A(t):=\mathcal{H}_{1}(A\cap S_t),
	\]
	where $\mathcal{H}_1$ is the one-dimensional Hausdorff measure and  $S_t:=\{(x,y)\in \Rp: x+y =t \}$.
	\end{defi}

	Clearly, $f_A$ is a measurable function. Moreover, for any Borell set  $A$~of~$\R_+^2$ we have 
	\begin{eqnarray*}
		\mu (A) = \int_{(x,y)\in A} e^{-(x+y)}dxdy = \int_{0}^\infty \int_{y: (t-y,y) \in A} e^{-t}dydt 			=\int_0^\infty \frac{1}{\sqrt{2}} f_A(t)e^{-t}dt.
	\end{eqnarray*}

	The next lemma introduces  a symmetrisation which  preserves  the function $f_A$. Moreover, the lemma states that  this symmetrisation does dot increase the boundary measure of the symmetrised set. This symmetrisation is illustrated on next two figures.
	
\begin{figure}[h]
\centering
\includegraphics[width=0.75\textwidth]{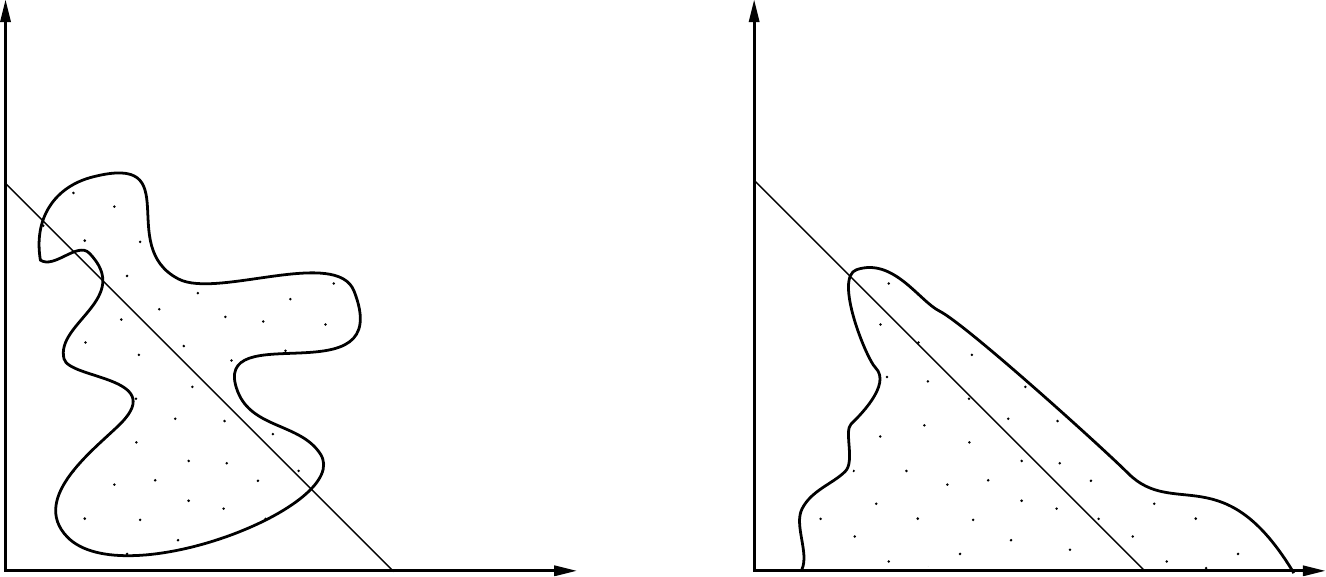}
\\ \hspace{0.03\textwidth} { \bf Fig. 1a.}  \hspace{0.33\textwidth} {\bf Fig. 1b.}
\\ \hspace{-0.1\textwidth} Set $A$ before the symmetrisation \hspace{0.22\textwidth} Set $C=C_A$ 
\end{figure}

\begin{lem} \label{lemat_red_1}
	For any Borel set $A\subset \Rp$ we introduce
	\[
		C=C_A:=\bigcup_{t>0} \left\{ (x,y)\in S_t: y\leqslant \frac{f_A(t)}{\sqrt{2}}  \right\} .
	\]
Then $\mu(C)=\mu(A)$ and $\mu (C+hT)\leqslant \mu (A+hT)$ for every $h>0$.
\end{lem}

\begin{proof}
By the definition of  $C$  we get $f_A=f_C$, so $\mu(A) = \mu(C)$. 

We will prove that for all $s,t,h>0$ we have
	\begin{eqnarray}\label{pom_5}
		\mathcal{H}_1 \big(S_s\cap (A\cap S_t +hT)\big) \geqslant \mathcal{H}_1 \big(S_s\cap (C\cap S_t+hT)\big).
	\end{eqnarray}
	This implies (since all the sets  $S_s\cap (C\cap S_t+hT) $ are intervals with endpoints at $(s,0)$) that $f_{A+hT} \ge f_{C+hT}$ and thus 
	\[
		\mu(A+hT) = \int_0^\infty \frac{1}{\sqrt{2}}f_{A+hT}(t) e^{-t} dt \ge \int_0^\infty \frac{1}{\sqrt{2}}f_{C+hT}(t) e^{-t} dt =\mu(C+hT)  .
	\]

Let us first consider the case $s\ge t$. It suffices to consider $h=s- t$, since for $h>s-t$ both  sides of \eqref{pom_5} do not change, while for $h<s-t$ both sides of \eqref{pom_5} vanish. Let $x$ be the point $(t-u, u)$, where $u$ is the smallest possible non-negative number such that the point $(t-u, u)$ belongs to $A$ (see Figure 2a). Since $h=s-t$, we have 
	\begin{equation*}
		(A\cap S_t+he_2) \cup ( x +S_h )  \subset S_s\cap (A\cap S_t)^h,
	\end{equation*}
	where $e_2=(0,1)$.

\begin{figure}[h]
\centering
\includegraphics[width=0.85\textwidth]{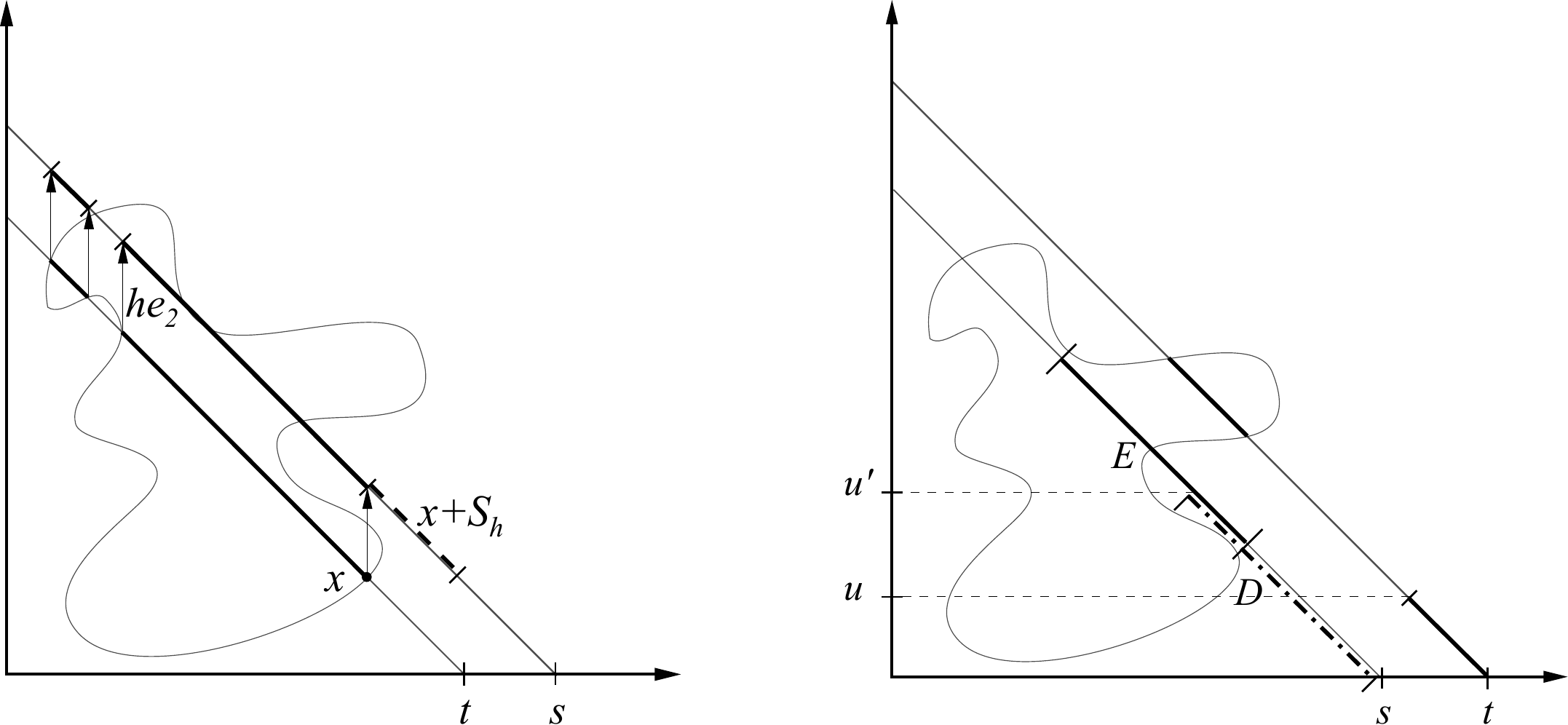}
\\ \hspace{-0.03\textwidth} { \bf Fig. 2a.}  \hspace{0.38 \textwidth} {\bf Fig. 2b.}
\end{figure}

	Moreover, this inclusion becomes an equality if we replace $A$ by $C$. Therefore 
	\begin{align*}
		\mathcal{H}_1 \big(S_s\cap (A\cap S_t+hT)\big) \geqslant \mathcal{H}_1(A\cap S_t) + \sqrt{2}h = \mathcal{H}_1(C\cap S_t) + \sqrt{2}h  =  \mathcal{H}_1 \big(S_s\cap (C\cap S_t)^h\big),
	\end{align*}
which shows that \eqref{pom_5} is satisfied in the case $t\le s$. 

 	Let us assume now that $t>s$. Again, it suffices to consider $h=t-s$. Suppose \eqref{pom_5} does not hold. Let $u'\ge 0$ be such that $E:=S_s\cap (A\cap S_t)^h $ has the same Hausdorff measure as $D:=S_s\cap \{(y_1,y_2) : y_2\leqslant u' \}$ (see Figure 2b). Let $u$ is given by  $C\cap S_t = S_t\cap\{(y_1,y_2): y_2\leqslant u \}$. Since \eqref{pom_5} does not hold, $u' < u \wedge s$ (note that in Figure 2b we have $u'\ge u$, since this figure reflects the true situation, whereas we are arguing by contradiction). By the conclusion of the first case (in which we had $t\leqslant s$),  we have
	\[
		\mathcal{H}_1 \big((S_s\setminus E+hT) \cap S_t \big) \geqslant \mathcal{H}_1 \big((S_s\setminus D+hT) \cap S_t \big) = \mathcal{H}_1 \big(\{(y_1,y_2):\ y_2\geqslant u' \} \cap S_t \big),
	\]
	since the sets $S_s\setminus E$ and $S_s \setminus D$ are of the same Hausdorff measure.

Moreover, by the definition of the set $E$ we get  $(S_s\setminus E +hT) \cap S_t \subset S_t\setminus A$. Therefore
	\begin{multline*}
		\mathcal{H}_1  ( A\cap S_t) \leqslant \mathcal{H}_1 \big( S_t \setminus \{(y_1,y_2): \ y_2\geqslant u' \} \big) = \mathcal{H}_1 \big( S_t \cap \{(y_1,y_2): \ y_2\leqslant u' \} \big)
		\\
		 < \mathcal{H}_1 \big( S_t \cap \{(y_1,y_2): \ y_2\leqslant u \} \big)  		= \mathcal{H}_1 (S_t\cap C), 
	\end{multline*}
which contradicts the property $\mathcal{H}_1  ( A\cap S_t) =  \mathcal{H}_1 (S_t\cap C)$. Hence \eqref{pom_5} is satisfied also in the case $s<t$. 
\end{proof}

Note that in higher dimensions the above proof works in the case $s\ge t$ (we only have to additionally use the Brunn-Minkowski inequality for an arbitrary set and a simplex). However, the same reasoning as above shows, that the analogue of \eqref{pom_5} for $s<t$ holds if we consider $\R_+^n \setminus D$ (where $\R_+^n \setminus D$ has the same measure as $C$, and $D$ is such that $C_D=D$) instead of $C$. Therefore \eqref{pom_5} fails in general if  $s<t$  and $n>2$. The reason why \eqref{pom_5} works for $n=2$ is that $S_t$ is an interval and therefore the sections of $C$ and $D$ (at the level $t$) are both intervals starting from an end point of $S_t$.

Now we are ready to prove the main theorem. Its proof clarifies, how to replace the set $C_A$ by a \emph{trapezoid}. This reasoning  fails in higher dimensions too. Also the induction over $n$ does not work, since a section parallel to the hyperplane $\operatorname{lin}(e_1,\ldots , e_{n-1})$  of a connected set does not have to be connected.

\begin{proof}[Proof of Theorem \ref{main_theorem}]
	Due to Lemma \ref{simplification_lemma} and Remark \ref{rem_T}  it suffices to prove $\nu(A+hT)\ge \nu(B_A+hT)-Lh^2$  for connected bounded compact sets $A$ and  for sufficiently (depending on $A$) small $h>0$. By Lemma \ref{lemat_red_1}  it suffices to prove that for sufficiently small $h$  the inequality $\nu(C+hT)\ge \nu(B_C+hT)-Lh^2$ holds for $C=C_A$. Let $f:=f_A=f_C$. 
	
	Note that for every Borel set $A$ and $h>0$ we have $\nu(A+he_1)=\nu(A+he_2)=e^{-h}\nu(A)$. Moreover, if $A - he_1\subset \Rp $ (or $A - he_2\subset \Rp $), then  $\nu(A-he_1)=e^h\nu(A)$ (or $\nu(A-he_2)=e^h\nu(A)$, respectively). We will use this observation throughout the proof.
	
	Recall that $C$ is compact and connected. Therefore, if for every $u>0$ we have $f(u)<\sqrt{2}u$ and $\operatorname{supp}f \subset (0,\infty)$, then there exists $\eps>0$ such that $f(u)<\sqrt{2}(u-\eps)$ for every $u>\eps$ and $f(u)=0$ for every $u\le \eps$ (this means that $C$ does not intersect the strip $[0,\eps)\times [0,\infty)$). Hence for every $h\in (0,\eps)$ we have $(C-he_1 )\subset C^h \cap \Rp$, where $e_1=(1,0)$, so 
	\[
	\nu (C^h)\ge \nu(C-he_1) = \nu(C)e^h = \nu(D)e^h \ge \nu (D+hT),
	\]
	 where $\nu(D)=\nu(C)$ and  $D$ is the complement of $rB_1^n$, and the last inequality follows since $D+hT \subset D-he_1$. Therefore we can restrict our attention to  the case in which there exists $u\ge 0$ such that $f(u)=\sqrt{2}u$. Let $u$ be the smallest value for which $\sqrt{2}u=f(u)$ (the minimal $u$ exists since $C$ is compact).
	
	Let $a\le u \le b$ be such that $\nu(uB_1^2\setminus aB_1^2) = \nu(C\cap uB_1^2)$ and $\nu(bB_1^2\setminus uB_1^2) = \nu(C\setminus uB_1^2)$. In other words we pick such $a$ and $b$, that the \emph{trapezoid} between $a$ and $u$ has the same measure as $C$ below $u$ (and similarly the \emph{trapezoid} between $u$ and $b$ has the same measure as $C$ above $u$).  We will show that 
	\begin{equation}\label{pom_6}
	\nu\big(C\cup (\Rp \setminus uB_1^2)+hT\big) \ge \nu\big(\Rp \setminus (a-h)B_1^n\big)  \quad \mbox{ for } 0<h< \min \{h_0,a\}
	\end{equation} 
	and 
	\begin{equation}\label{pom_7}
	\nu(C\cup uB_1^2+hT) \ge \nu\big((b+h)B_1^2 \big) - Lh^2 \quad \mbox{ for } h>0,
	\end{equation}
	 where $h_0:=\min\left\{ \max\{ \lambda: \nu_1(R_\lambda) \le \nu_1([a,u]) \} , h_1 \right\}$ and $R_{\lambda} :=\{t\in (\lambda,u): t- \frac{f(t)}{\sqrt{2}}<\lambda \}$, and $\nu_1$ is the marginal distribution of $\nu$, i.e. the exponential measure on the half-line, and $h_1$ is such that $\operatorname{supp} f\subset (h_1,\infty)$. By Lemma \ref{lem_trap}, inequalities \eqref{pom_6} and \eqref{pom_7} will finish the proof of the theorem (we will see below that $h_0>0$ if $a\neq 0$), since \eqref{pom_6} say that the  $hT$ neighbourhood of $C$ below $u$ is not less than  the  $hT$ neighbourhood of the \emph{trapezoid} between $a$ and $u$ (and similarly \eqref{pom_7} gives us an analogous estimate above $u$, up to a term $Lh^2$).

	Let us first show  \eqref{pom_6} (Figure 3a is attached for the reader's convenience). If $a=0$ or $a=u$, we have nothing to prove. Suppose therefore that $u>a>0$ and $h<h_0 $.  Note that $h_0>0$, since $0<a<u$ and for every $0<t<u$ we have $f(t)<\sqrt{2}t$. 
	
\begin{figure}[h]
\centering
\includegraphics[width=0.85\textwidth]{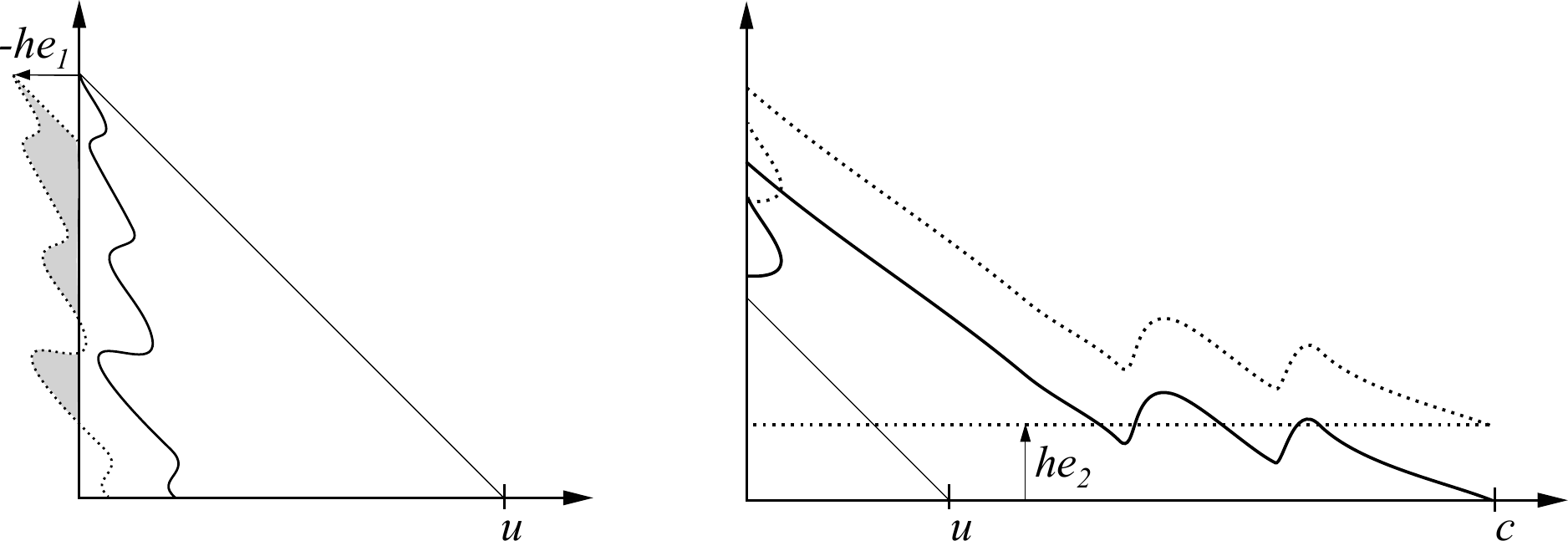}
\\ \hspace{-0.05\textwidth} { \bf Fig. 3a.} \qquad  \qquad \qquad \qquad  \hspace{0.17\textwidth}{\bf Fig. 3b.}
\end{figure}

	Moreover,   
	\[
	\big((C \cap   uB_1^2 -he_1)\cap \Rp\big) \cup \big(\Rp \setminus (u-h)B_1^n\big) \subset \big(C\cup (\Rp \setminus  uB_1^2)\big) +hT 
	\]
	 and, since $h\le h_1$,  the set
	 $\Rp\setminus (C \cap  uB_1^2-he_1 ) $ (see  the grey set in Figure 3a) is contained in the set $ \bigcup_{\delta=0}^h \big( \{0 \}\times (R_h-h) +(-\delta, \delta)\big)$, which is the translation by the vector $-he_1$ of a set of measure $h\nu_1(R_h)$. By the definition of $h_0$ we know that for $h\le h_0$ we  have $\nu_1(R_h) \le \nu_1([a,u]) $. Therefore
	\begin{align*}
		\nu\big((C\cup (\Rp \setminus & uB_1^2)) +hT  \big)
		\\ &
		  \ge  \nu\big(\Rp  \setminus (u-h)B_1^2\big)  + e^h\nu (C\cap uB_1^n )
		 - e^h h \cdot \nu_1(R_h) 
		 \\ &
		 \ge  \nu\big(\Rp \setminus (u-h)B_1^2\big) + e^h\nu ( uB_1^n \setminus aB_1^n ) 
		 - e^h h  \cdot \nu_1([a,u])
		 \\ & =  \nu\big(\Rp \setminus (a-h)B_1^n\big) ,
	\end{align*}
	what yields inequality \eqref{pom_6}.

	 We will prove inequality \eqref{pom_7} (Figure 3b may be helpful to follow the estimates). Note that the fact that  $C$ is connected implies that $\operatorname{supp} f $ is connected, and let $c:=\sup \operatorname{supp} f$. Obviously $c\ge b$ and $\operatorname{supp} f \cup [0,u]=[0,c]$. Moreover we have $ ((C\cup uB_1^2) + he_2) \cup [0,c]\times [0,h] \subset (C\cup uB_1^2)+hT$, so 
	\begin{multline*}
		\nu\big((C\cup uB_1^2) +hT\big) \ge \nu\big( ((C\cup uB_1^2) + he_2) \cup [0,c]\times [0,h]\big) 
		\ge \nu\big( (C\cup uB_1^2) + he_2)+ \nu(  [0,b]\times [0,h] \big)
		\\
		= \nu(bB_1^n+he_2) +  \nu(  [0,b]\times [0,h] \big)  \ge \nu\big(  (b+h)B_1^n\big) -Lh^2,
	\end{multline*}
	where $L$ is an absolute constant. The proof of the theorem is finished.
\end{proof}

{\bf Acknowledgements.} 
 I would like to thank Piotr Nayar and Tomasz Tkocz for introducing me to the topic of isoperimetric problems and posing an  inspiring question about examples of  extremal sets for the geometric measure on $\mathbb{Z}_+^2$.

\end{document}